\documentclass[11pt,amssymb]{amsart}
\usepackage{graphicx}
\usepackage{amsmath,amssymb,amsfonts,euscript,enumerate}
\usepackage{amsthm}
\usepackage{color,wrapfig}
\usepackage{verbatim}
\newtheorem{thm}{Theorem}[section]
\newtheorem{cor}[thm]{Corollary}
\newtheorem{lemma}[thm]{Lemma}
\newtheorem{prop}[thm]{Proposition}

\newcommand{\R}{{\mathbb{R}}}
\newcommand{\Z}{{\mathbb{Z}}}

\newcommand{\La}{\triangle}
\newcommand{\bs}{\backslash}

\newcommand{\1}{\partial}

\newcommand{\3}{\varepsilon}
\newcommand{\4}{\widetilde}

\begin{document}
\title{Decay rate and radial symmetry of the exponential elliptic equation}
\author[Kin Ming Hui]{Kin Ming Hui}
\address{Kin Ming Hui:
Institute of Mathematics, Academia Sinica,\\
Taipei, Taiwan, R.O.C.}
\email{kmhui@gate.sinica.edu.tw}
\author[Sunghoon Kim]{Sunghoon Kim}
\address{Sunghoon Kim:
Department of Mathematics and PMI (Pohang Mathematics Institute), 
Pohang University of Science and Technology (POSTECH),\\
Hyoja-Dong San 31, Nam-gu, Pohang 790-784, South Korea}
\email{math.s.kim@postech.ac.kr}
\date{Sept 3, 2012}
\maketitle
\begin{abstract}
Let $n\ge 3$, $\alpha$, $\beta\in\R$, and let $v$ be a solution 
$\Delta v+\alpha e^v+\beta x\cdot\nabla e^v=0$ in $\R^n$, which satisfies 
the conditions $\lim_{R\to\infty}\frac{1}{\log R}\int_{1}^{R}\rho^{1-n}
\left(\int_{B_{\rho}}e^v\,dx\right)d\rho\in (0,\infty)$ and 
$|x|^2e^{v(x)}\le A_1$ in $\R^n$. We prove that $\frac{v(x)}{\log |x|}\to -2$ 
as $|x|\to\infty$ and $\alpha>2\beta$. As a consequence under a mild 
condition on $v$ we prove that the solution is radially symmetric 
about the origin.  
\end{abstract}
\vskip 0.2truein

\setcounter{equation}{0}
\setcounter{section}{0}

\section{Introduction}
\setcounter{equation}{0}
\setcounter{thm}{0}

In this paper we will study various properties of the 
solution $v$ of the nonlinear elliptic equation

\begin{equation}\label{v-eqn}
\Delta v+\alpha e^v+\beta x\cdot\nabla e^v=0\quad\mbox{ in }\R^n
\end{equation}
for any $n\ge 3$ where $\alpha$, $\beta\in\R$, are some constants. Let
$v=\log u$. Then $u$ satisfies 
\begin{equation}\label{u-eqn}
\Delta\log u+\alpha u+\beta x\cdot\nabla u=0, \quad u>0,\quad\mbox{ in }\R^n.
\end{equation}
As observed by S.Y.~Hsu \cite{Hs3}, the radial symmetric solution of 
\eqref{u-eqn} is the singular limit of the radial symmetric solutions 
of the nonlinear elliptic equation,
\begin{equation}\label{m-elliptic-eqn}
\Delta (u^m/m)+\alpha u+\beta x\cdot\nabla u=0, \quad u>0,\quad\mbox{ in }
\R^n,
\end{equation}
as $m\searrow 0$. On the other hand as observed by P.~Daskalopoulos and 
N.~Sesum \cite{DS}, K.M.~Hui and S.H.~Kim \cite{HK1}, \cite{HK2}, 
\eqref{u-eqn} also arises in the study of the extinction behaviour and 
global behaviour of the solutions of the logarithmic diffusion equation,
\begin{equation}\label{log-eqn}
u_t=\Delta \log u,\quad u>0,\quad\mbox{ in }\R^n.
\end{equation}
\eqref{u-eqn} also arises in the study of self-similar solutions of 
\eqref{log-eqn} (\cite{DS}, \cite{HK1}, \cite{HK2}, \cite{V1}, \cite{V2}).
Hence in order to understand the behaviour of the solutions of 
\eqref{m-elliptic-eqn} and \eqref{log-eqn} it is important to understand
the properties of solutions of \eqref{v-eqn}.

In \cite{Hs2} S.Y.~Hsu proved that there exists a radially symmetric solution 
of \eqref{v-eqn} ( or equivalently \eqref{u-eqn}) if and only if either 
$\alpha\ge 0$ or $\beta>0$. She also proved that when $n\ge 3$ and 
$\alpha>\max (2\beta,0)$, then any radially symmetric solution $v$ of 
\eqref{v-eqn} satisfies
\begin{equation}\label{e^v-limit}
\lim_{|x|\to\infty}|x|^2e^{v(x)}=\frac{2(n-2)}{\alpha-2\beta}.
\end{equation} 
By \eqref{e^v-limit} and a direct computation the radially symmetric 
solution $v$ of \eqref{v-eqn} satisfies
\begin{equation}\label{v-limit}
\lim_{|x|\to\infty}\frac{v(x)}{log |x|}=-2,
\end{equation} 
\begin{equation}\label{e^v-integral-cond}
A_0:=\lim_{R\to\infty}\frac{1}{\log R}\int_{1}^{R}\frac{1}{\rho^{n-1}}
\left(\int_{|x|<\rho}e^v\,dx\right)d\rho\in (0,\infty)
\end{equation}
and
\begin{equation}\label{e^v-bd}
|x|^2e^{v(x)}\le A_1\quad\forall x\in\R^n
\end{equation}
for some constant $A_1>0$. A natural question is if $v$ is a solution of
\eqref{v-eqn} which satisfies \eqref{e^v-integral-cond} 
and \eqref{e^v-bd} for some constant $A_1>0$, will $v$ satisfy \eqref{v-limit}
and is $v$ radially symmetric about the origin? We answer the first
question in the affirmative in this paper. For the second question we 
prove that under some conditions on the solution $v$ of \eqref{v-eqn},
$v$ is radially symmetric about the origin. 

For any solution $v$ of \eqref{v-eqn} we define the rotation operator 
$\Phi_{ij}$ by
\begin{equation*}
\Phi_{ij}(x)=x_iv_{x_j}(x)-x_jv_{x_i}(x), \qquad \forall x=(x_1,\dots,x_n)
\in\R^n,i\ne j,i,j=1,\cdots,n.
\end{equation*}
Note that if we write $x_1=\rho\cos\theta$ and $x_2=\rho\sin\theta$ where
$\rho=\sqrt{x_1^2+x_2^2}$, then $\Phi_{12}(x)=\frac{\1 v}{\1\theta}(x)$.
We are now ready to state the main results of this paper.

\begin{thm}\label{v-limit-thm}
Let $n\ge 3$ and $\alpha,\beta\in\R$. Suppose $v$ is a solution of 
\eqref{v-eqn} which satisfies \eqref{e^v-integral-cond} and \eqref{e^v-bd} 
for some constant $A_1>0$. Then $v$ satisfies \eqref{v-limit} and 
$\alpha>2\beta$.
\end{thm}

\begin{cor}\label{non-existence-cor}
Let $n\ge 3$. Suppose $\alpha\le 2\beta$. Then \eqref{v-eqn} does not 
have any solution that satisfies both \eqref{e^v-integral-cond} and 
\eqref{e^v-bd} for some constant $A_1>0$.
\end{cor}

\begin{thm}\label{radial-symmetric-thm}
Let $n\ge 3$ and $2\beta<\alpha<n\beta$. Suppose $v$ is a solution of 
\eqref{v-eqn} which satisfies \eqref{e^v-integral-cond}, \eqref{e^v-bd}, 
\begin{equation}\label{grad-v-bd-assumption}
\left\|x\cdot\nabla v\right\|_{L^{\infty}(\R^n)}\leq C<\infty.
\end{equation}
and
\begin{equation}\label{radial-symm-assum-infty}
\lim_{|x|\to\infty}|x|^{n-2}\left|\Phi_{ij}(x)\right|=0,\quad 
\forall i\ne j, i,j=1,\cdots,n.
\end{equation}
Then there exists a constant $R_0>0$ such that if $v$ is radially symmetric 
in $B_{R_0}$, then $v$ is radially symmetric in $\R^n$.
\end{thm}

Note that although there are many research done on the radial symmetry of
elliptic equations without first order term by B.~Gidas, W.M.~Ni, and 
L.~Nirenberg \cite{GNN}, L.~Caffaralli, B.~Gidas and J.~Spruck \cite{CGS},
W.~Chen and C.~Li \cite{CL}, S.D.~Taliaferro \cite{T} and others, very 
little is known about the radial symmetry of elliptic equations with 
non-zero first order term.
The reason is that one cannot use the moving plane technique to prove the 
radial symmetry of the solution for elliptic equations with non-zero first 
order term. The recent paper \cite{KM} by E.~Kamalinejad and A.~Moradifam
is one of the few papers that studies the radial symmetry of
elliptic equations with non-zero first order term. Hence our result on
radial symmetry is new.

The plan of the paper is as follows. In section two we will prove 
Theorem \ref{v-limit-thm} and Corollary \ref{non-existence-cor}. In 
section three we will prove Theorem \ref{radial-symmetric-thm}. 
For any $r>0$, $x_0\in\R^n$, let $B_r(x_0)=\{x\in\R^n:|x-x_0|<r\}$
and $B_r=B_r(0)$. Let $S^{n-1}=\{x\in\R^n:|x|=1\}$. We will let $n\ge 3$, 
$\alpha, \beta\in\R$, and let 
$v$ be a solution of \eqref{v-eqn} which satisfies both 
\eqref{e^v-integral-cond} and \eqref{e^v-bd} for some constant $A_1>0$ 
for the rest of the paper. We will also let $A_0$ be given by 
\eqref{e^v-integral-cond} for the rest of the paper.

\section{Decay rate of the solution of \eqref{v-eqn}}
\setcounter{equation}{0}
\setcounter{thm}{0}

In this section we will use a modification of the technique of
\cite{Hs1} to prove the decay rate \eqref{v-limit} for $v$. We first 
start with a lemma.

\begin{lemma}
For any $x_0\in\R^n$, we have
\begin{equation}\label{eq-general-of-condition-1}
\lim_{R\to\infty}\frac{1}{\log R}\int_{1}^{R}\frac{1}{\rho^{n-1}}
\left(\int_{B_{\rho}(x_0)}e^v\,dx\right)d\rho=A_0.
\end{equation}
\end{lemma}
\begin{proof}
Let $x_0\in\R^n$ and $\rho>|x_0|$. Since $B_{\rho-|x_0|}\subset B_{\rho}(x_0)$,
for any $R>R_0>|x_0|+1$, 
\begin{equation*}
\begin{aligned}
&\frac{1}{\log R}\int_{1}^{R}\frac{1}{\rho^{n-1}}
\left(\int_{B_{\rho}(x_0)}e^v\,dx\right)d\rho\\
&\quad \ge\frac{1}{\log R}\int_{R_0}^{R}\frac{1}{\rho^{n-1}}
\left(\int_{B_{\rho-|x_0|}}e^v\,dx\right)d\rho\\
&\quad \ge\frac{1}{\log R}\left(\frac{R_0-|x_0|}{R_0}\right)^{n-1}
\int_{R_0-|x_0|}^{R-|x_0|}\frac{1}{\rho^{n-1}}\left(\int_{B_{\rho}}e^v\,dx\right)
d\rho\\
&\quad =-\frac{1}{\log R}\left(\frac{R_0-|x_0|}{R_0}\right)^{n-1}\int_{1}^{R_0-|x_0|}
\frac{1}{\rho^{n-1}}\left(\int_{B_{\rho}}e^v\,dx\right)d\rho\\
&\qquad +\frac{\log (R-|x_0|)}{\log R}\left(\frac{R_0-|x_0|}{R_0}\right)^{n-1}
\left[\frac{1}{\log(R-|x_0|)}\int_{1}^{R-|x_0|}\frac{1}{\rho^{n-1}}
\left(\int_{B_{\rho}}e^v\,dx\right)d\rho\right].
\end{aligned}
\end{equation*}
Letting $R\to\infty$, 
\begin{align}
&\liminf_{R\to\infty}\frac{1}{\log R}\int_{1}^{R}\frac{1}{\rho^{n-1}}
\left(\int_{B_{\rho}(x_0)}e^v\,dx\right)d\rho\geq \left(\frac{R_0-|x_0|}{R_0}
\right)^{n-1}A_0\notag\\
&\qquad \Rightarrow \liminf_{R\to\infty}\frac{1}{\log R}
\int_{1}^{R}\frac{1}{\rho^{n-1}}\left(\int_{B_{\rho}(x_0)}e^v\,dx\right)
d\rho\geq A_0\qquad\mbox{ as }R_0\to\infty
\label{liminf-e^v-integral-lower-bd}
\end{align}
Similarly,
\begin{equation}\label{limsup-e^v-integral-upper-bd}
\limsup_{R\to\infty}\frac{1}{\log R}\int_{1}^{R}\frac{1}{\rho^{n-1}}
\left(\int_{B_{\rho}(x_0)}e^v\,dx\right)d\rho\leq A_0.
\end{equation}
By \eqref{liminf-e^v-integral-lower-bd} and 
\eqref{limsup-e^v-integral-upper-bd}, we get 
\eqref{eq-general-of-condition-1} and the lemma follows.
\end{proof}

\begin{lemma}\label{lem-limit-of-v-as-rho-to-infty}
For any $x_0\in\R^n$, we have
\begin{equation}\label{e^v-integral-mean}
\lim_{\rho\to\infty}\rho^{2}\int_{|\sigma|=1}e^{v(x_0+\rho\sigma)}\,d\sigma 
=(n-2)A_0
\end{equation}
and
\begin{equation}\label{v-integral-log-mean}
\lim_{\rho\to\infty}\left[\frac{1}{\log\rho}\int_{|\sigma|=1}v(x_0+\rho\sigma)
\,d\sigma\right]=-(\alpha-2\beta)A_0.
\end{equation}
\end{lemma}
\begin{proof}
Let $x_0\in\R^n$ and $\rho>0$. We first observe that \eqref{v-eqn} can be
rewritten as
\begin{equation}\label{v-eqn21}
\Delta v+(\alpha-n\beta)e^v+\beta\,\mbox{div}(xe^v)=0\quad\mbox{ in }\R^n.
\end{equation}
Integrating \eqref{v-eqn21} over $B_{\rho}(x_0)$, 
\begin{equation}\label{v-integral-eqn}
\begin{aligned}
0=&\rho^{n-1}\int_{|\sigma|=1}\frac{\partial v}{\partial \rho}(x_0+\rho\sigma)
\,d\sigma+(\alpha-n\beta)\int_{B_{\rho}(x_0)}e^v\,dx+\beta\rho^n
\int_{|\sigma|=1}e^{v(x_0+\rho\sigma)}\,d\sigma\\
&\qquad +\beta\rho^{n-1}\int_{|\sigma|=1}e^{v(x_0+\rho\sigma)}
\left(x_0\cdot\sigma\right)d\sigma.
\end{aligned}
\end{equation}
Let $R>R_0\geq 1$. Dividing \eqref{v-integral-eqn} by $\rho^{n-1}$ and 
integrating over $\rho\in(1,R)$,
\begin{equation}\label{eq-aligned-integrating-over-1-to-R-1}
\begin{aligned}
&\int_{|\sigma|=1}v(x_0+R\sigma)\,d\sigma-\int_{|\sigma|=1}v(x_0+\sigma)\,d\sigma\\
&\qquad =-(\alpha-n\beta)\int_{1}^{R}\frac{1}{\rho^{n-1}}
\left(\int_{B_{\rho}(x_0)}e^v\,dx\right)d\rho-\beta\int_{1}^{R}\rho
\left(\int_{|\sigma|=1}e^{v(x_0+\rho\sigma)}\,d\sigma\right)d\rho\\
&\qquad \qquad -\beta\int_{1}^{R}\left(\int_{|\sigma|=1}e^{v(x_0+\rho\sigma)}
(x_0\cdot\sigma)\,d\sigma\right)d\rho.
\end{aligned}
\end{equation}
Let $\{\rho_i\}_{i=1}^{\infty}$ be a sequence such that $\rho_i>2|x_0|+1$ for 
all $i\in\Z^+$ and $\rho_i\to\infty$ as $i\to\infty$. Then by \eqref{e^v-bd}
there exists a constant $C_1>0$ such that
\begin{equation*}
\rho_i^2\int_{|\sigma|=1}e^{v(x_0+\rho_i\sigma)}\,d\sigma
\le C_1\quad\forall i\in\Z^+.
\end{equation*}
Hence the sequence $\{\rho_i\}_{i=1}^{\infty}$ has a subsequence which we
may assume without loss of generality to be the sequence itself such that 
$$
\rho_i^2\int_{|\sigma|=1}e^{v(x_0+\rho_i\sigma)}\,d\sigma
$$ converges to some non-negative number as $i\to\infty$. On the other
hand by \eqref{eq-general-of-condition-1}, there exists a constant $R_1>1$
such that
\begin{align}
\frac{A_0}{2}\le&\frac{1}{\log R}\int_1^R\frac{1}{\rho^{n-1}}
\left(\int_{B_{\rho}(x_0)}e^v\,dx\right)\,d\rho\le\frac{1}{(n-2)\log R}
\int_{B_R(x_0)}e^v\,dx\quad\forall R\ge R_1\notag\\
\Rightarrow\quad\int_{B_{\rho}(x_0)}e^v\,dx\ge&\frac{(n-2)}{2}A_0\log\rho\to\infty
\quad\mbox{ as }\rho\to\infty.\label{e^v-integral=infty}
\end{align}
Then by \eqref{e^v-integral-cond}, \eqref{eq-general-of-condition-1}, 
\eqref{e^v-integral=infty}, and the l'Hospital rule,
\begin{equation*}
\begin{aligned}
A_0&=\lim_{\rho\to\infty}\frac{\frac{1}{\rho_i^{n-1}}\int_{B_{\rho_i}(x_0)}e^v\,dx}
{\frac{1}{\rho_i}}=\lim_{i\to\infty}\frac{1}{\rho_i^{n-2}}
\int_{B_{\rho_i}(x_0)}e^v\,dx\\
&=\lim_{i\to\infty}\frac{\rho_i^{n-1}\int_{|\sigma|=1}e^{v(x_0+\rho_i\sigma)}\,d\sigma}
{(n-2)\rho_i^{n-3}}=\frac{1}{(n-2)}\lim_{i\to\infty}\rho_i^{2}
\int_{|\sigma|=1}e^{v(x_0+\rho_i\sigma)}\,d\sigma
\end{aligned}
\end{equation*}
Since the sequence $\{\rho_i\}_{i=1}^{\infty}$ is arbitrary, 
\eqref{e^v-integral-mean} follows. Then by \eqref{e^v-integral-mean}
for any $0<\3<1$ there exists a constant $R_0>1$ such that
\begin{equation}\label{eq-bounded-of-surface-integral-e-v-over-rho-2---}
\left|\rho^2\int_{|\sigma|=1}e^{v(x_0+\rho\sigma)}d\sigma-(n-2)A_0\right|
<(n-2)A_0\3 \qquad \forall \rho\geq R_0.
\end{equation}
By \eqref{eq-bounded-of-surface-integral-e-v-over-rho-2---},
\begin{equation}\label{r-e^v-integral-lower-upper-bd}
\begin{aligned}
&\int_1^{R_0}\rho\left(\int_{|\sigma|=1}e^{v(x_0+\rho\sigma)}\,d\sigma\right)
d\rho+(1-\3)(n-2)A_0\log (R/R_0)\\
&\qquad\le\int_1^R\rho\left(\int_{|\sigma|=1}e^{v(x_0+\rho\sigma)}
\,d\sigma\right)d\rho\\
&\qquad\qquad
\le\int_1^{R_0}\rho\left(\int_{|\sigma|=1}e^{v(x_0+\rho\sigma)}\,d\sigma\right)
d\rho+(1+\3)(n-2)A_0\log (R/R_0)\quad\forall R>R_0
\end{aligned}
\end{equation}
and
\begin{equation}\label{e^v-surface-mean-bd1}
\begin{aligned}
&\left|\int_{1}^{R}\int_{|\sigma|=1}e^{v(x_0+\rho\sigma)}(x_0\cdot\sigma)
\,d\sigma d\rho\right|\\
&\qquad \leq |x_0|\int_{1}^{R_0}
\left(\int_{|\sigma|=1}e^{v(x_0+\rho\sigma)}\,d\sigma\right)d\rho
+2(n-2)A_0|x_0|\left(\frac{1}{R_0}-\frac{1}{R}\right)
\end{aligned}
\end{equation}
for any $R>R_0$. Dividing \eqref{r-e^v-integral-lower-upper-bd}
and \eqref{e^v-surface-mean-bd1} by $\log R$ and letting first 
$R\to\infty$ and then $\3\to 0$, we get
\begin{equation}\label{r-e^v-integral/log-r-limit}
\lim_{R\to\infty}\frac{1}{\log R}\int_1^R\rho
\left(\int_{|\sigma|=1}e^{v(x_0+\rho\sigma)}\,d\sigma\right)d\rho=(n-2)A_0 
\end{equation}
and
\begin{equation}\label{e^v-surface-mean-limit-1}
\lim_{R\to\infty}\frac{1}{\log R}\int_1^R\int_{|\sigma|=1}e^{v(x_0+\rho\sigma)}
(x_0\cdot\sigma\,d\sigma d\rho=0.
\end{equation}
Dividing \eqref{eq-aligned-integrating-over-1-to-R-1} by $\log R$ and letting 
$R\to\infty$, by \eqref{e^v-integral-cond}, \eqref{r-e^v-integral/log-r-limit}
and \eqref{e^v-surface-mean-limit-1}, we get \eqref{v-integral-log-mean} 
and the lemma follows.
\end{proof}

We now let
\begin{equation}\label{eq-cases-aligned-newtonian-potentials}
\begin{cases}
\begin{aligned}
w_1(x)&=\frac{1}{n(2-n)\omega_n}\int_{\R^n}
\left(\frac{1}{|x-y|^{n-2}}-\frac{1}{|y|^{n-2}}\right)e^{v(y)}\,dy \qquad 
\forall x\in\R^n\\
w_{2,R}(x)&=\frac{1}{n(2-n)\omega_n}\int_{|y|\leq R}
\left(\frac{1}{|x-y|^{n-2}}-\frac{1}{|y|^{n-2}}\right)\mbox{div}
\left(e^{v(y)}y\right)\,dy  \qquad \forall x\in B_R,\,\,R>0
\end{aligned}
\end{cases}
\end{equation}
where $\omega_n$ is the volume of the unit ball in $\R^n$. Then $w_1$ and 
$w_{2,R}$ are well-defined in $\R^n$ and $B_R$ respectively with 
$w_1\in C(\R^n)$ and $w_{2,R}\in C(B_R)$. Moreover
\begin{equation}\label{eq-cases-aligned-laplace-of-w1-and-w2}
\begin{cases}
\begin{aligned}
\La w_1&=e^v \qquad \qquad \qquad \qquad \mbox{in }\R^n\\
\La w_{2,R}&=\mbox{div}\left(e^{v(y)}y\right) \qquad \quad 
\mbox{ in }B_R\quad\forall R>0.
\end{aligned}
\end{cases}
\end{equation}

\begin{lemma}\label{lem-convergence-of-w2-w-r-t-R-to-infty}
As $R\to\infty$, $w_{2,R}$ will converge uniformly on every compact subset 
of $\R^n$ to
\begin{equation}\label{eq-aligend-w-2-after-integration-by-part}
\begin{aligned}
w_2(x)&=\frac{1}{n(n-2)\omega_n}\int_{\R^n}e^{v(y)}\nabla_y
\left(\frac{1}{|x-y|^{n-2}}-\frac{1}{|y|^{n-2}}\right)\cdot y\,dy\\
&=\frac{1}{n\omega_n}\int_{\R^n}\left(\frac{(x-y)\cdot y}{|x-y|^n}
+\frac{1}{|y|^{n-2}}\right)e^{v(y)}\,dy.
\end{aligned}
\end{equation}
\end{lemma}

\begin{proof}
Let $R_0>1$, $R>2R_0$ and $|x|\le R_0$. Now
\begin{equation}\label{eq-aligned-seperation-of-w-2-into-two-parts-1}
\begin{aligned}
w_{2,R}(x)=&\frac{1}{n(2-n)\omega_n}\int_{|y|=R}\left(\frac{1}{|x-y|^{n-2}}
-\frac{1}{|y|^{n-2}}\right)e^{v(y)} y\cdot \nu(y)\,d\sigma_R(y)\\
&\qquad +\frac{1}{n(n-2)\omega_n}\int_{|y|<R}e^{v(y)}\nabla_y
\left(\frac{1}{|x-y|^{n-2}}-\frac{1}{|y|^{n-2}}\right)\cdot y\,dy\\
=&I_{1,R}(x)+I_{2,R}(x)
\end{aligned}
\end{equation}
where $\nu(y)=\frac{y}{|y|}$. Since there exists a constant $C_1>0$
such that 
\begin{align}\label{ratios-difference-eqn}
\left|\frac{1}{|x-y|^{n-2}}-\frac{1}{|y|^{n-2}}\right|
&=\left|\int_{0}^{1}\frac{\partial}{\partial\theta}
\left(\frac{1}{|\theta x-y|^{n-2}}\right)d\theta\right|
=(n-2)\left|\int_{0}^{1}\frac{(\theta x-y)\cdot x}{|\theta x-y|^n}\,d\theta
\right|\notag\\
&\le C_1\frac{|x|}{|y|^{n-1}} \qquad \forall |y|\ge 2|x|,
\end{align}
by \eqref{eq-bounded-of-surface-integral-e-v-over-rho-2---} there exists 
a constant $C_2>0$ such that
\begin{equation}\label{eq-convergence-to-zero-of-I-1-R-204859}
\sup_{|x|<R_0}\left|I_{1,R}(x)\right|\leq \frac{C_2R_0}{R} \to 0 \qquad 
\mbox{as }R\to\infty.
\end{equation}
By the Taylor expansion (P.231 of \cite{GNN}),
\begin{equation*}
\frac{1}{|x-y|^n}=\frac{1}{|y|^n}\left(1+\frac{n}{|y|^2}\sum_{j=1}^nx_jy_j
+O\left(\frac{1}{|y|^2}\right)\right)\quad \forall |y|\ge 2|x|,
\,\,x=(x_1\cdots,x_n),\,\,y=(y_1,\cdots,y_n).
\end{equation*}
Hence there exists a constant $C_3>0$ such that 
\begin{align}\label{eq-estimate-of-F-by-using-Taylor-expansion-F1}
&\left|\nabla_y\left(\frac{1}{|x-y|^{n-2}}-\frac{1}{|y|^{n-2}}\right)
\cdot y\right|\notag\\
&\qquad =(n-2)\left|\frac{(x-y)\cdot y}{|x-y|^n}+\frac{1}{|y|^{n-2}}\right|\notag\\
&\qquad =(n-2)\left|\frac{x\cdot y}{|y|^n}\left(1+\frac{n}{|y|^2}\sum_{j=1}^nx_jy_j
+\cdots\right)-\frac{1}{|y|^{n-2}}\left(\frac{n}{|y|^2}\sum_{j=1}^nx_jy_j
+\cdots\right)\right|\notag\\
&\qquad \le C_3\frac{|x|}{|y|^{n-1}}, \qquad\forall |y|\ge 2|x|,
\,\,x=(x_1\cdots,x_n),\,\,y=(y_1,\cdots,y_n).
\end{align}
By 
\eqref{eq-bounded-of-surface-integral-e-v-over-rho-2---} and
\eqref{eq-estimate-of-F-by-using-Taylor-expansion-F1},
\begin{equation}\label{eq-aligned-estimate-of-w-2-over-outside-of-2-R-0-1}
\begin{aligned}
&\int_{2R_0<|y|<R}e^{v(y)}\left|\nabla_y\left(\frac{1}{|x-y|^{n-2}}
-\frac{1}{|y|^{n-2}}\right)\cdot y\right|\,dy\\
&\qquad <\int_{2R_0}^{R}\frac{C_3R_0}{\rho^2}\left(
\rho^2\int_{|\sigma|=1}e^{v(\rho\sigma)}\,d\sigma\right)\,d\rho \le C'R_0\int_{2R_0}^{R}\frac{1}{\rho^2}\,d\rho\le C''<\infty
\end{aligned}
\end{equation}
holds for any $|x|\le R_0$ and $R>2R_0$. On the other hand since 
$v$ is continuous 
and $v$ satisfies \eqref{e^v-bd}, $e^v\in L^{\infty}(\R^n)$. Hence by 
\eqref{eq-estimate-of-F-by-using-Taylor-expansion-F1},
\begin{equation}\label{eq-aligned-estimate-of-w-2-over-inside-of-2-R-0-1}
\begin{aligned}
&\int_{|y|<2R_0}e^{v(y)}\left|\nabla_y
\left(\frac{1}{|x-y|^{n-2}}-\frac{1}{|y|^{n-2}}\right)\cdot y\right|\,dy\\
&\qquad \qquad \leq C_3R_0\int_{|y|<3R_0}\frac{1}{|y|^{n-1}}\,dy
=3n\omega_nC_3R_0^2<\infty.
\end{aligned}
\end{equation}
By
\eqref{eq-aligned-estimate-of-w-2-over-outside-of-2-R-0-1} and 
\eqref{eq-aligned-estimate-of-w-2-over-inside-of-2-R-0-1} and the Lebesgue 
Dominated Convergence Theorem,
\begin{equation}\label{eq-convergence-to-desired-value-of-I-2-R-1204859}
I_{2,R}(x)\to \frac{1}{n\omega_n}\int_{\R^n}
\left(\frac{(x-y)\cdot y}{|x-y|^n}+\frac{1}{|y|^{n-2}}\right)e^{v(y)}\,dy
\end{equation}
uniformly on $\{|x|\leq R_0\}$ as $R\to\infty$. By 
\eqref{eq-aligned-seperation-of-w-2-into-two-parts-1}, 
\eqref{eq-convergence-to-zero-of-I-1-R-204859} and 
\eqref{eq-convergence-to-desired-value-of-I-2-R-1204859}, the lemma follows.
\end{proof}

\begin{lemma}\label{lem-estimate-for-w-1-by-I-1-and-I-2}
There exists a constant $C>0$ such that
\begin{equation}\label{eq-bounded-of-w-1-by-log-function}
\left|w_1(x)\right|\le C\log|x|\qquad \forall |x|\ge 2.
\end{equation}
\end{lemma}
\begin{proof}
Let $|x|\ge 2$. We first split $w_1$ into two parts as follows.
\begin{align}\label{w1-split-eqn}
w_1(x)&=\frac{1}{n(2-n)\omega_n}\int_{|y|>2|x|}
\left(\frac{1}{|x-y|^{n-2}}-\frac{1}{|y|^{n-2}}\right)e^{v(y)}\,dy\notag\\
&\qquad +\frac{1}{n(2-n)\omega_n}\int_{|y|\le 2|x|}\frac{e^{v(y)}}{|x-y|^{n-2}}\,dy
+\frac{1}{n(n-2)\omega_n}\int_{|y|\le 2|x|}\frac{e^{v(y)}}{|y|^{n-2}}\,dy\notag\\
&:=I_1+I_2+I_3.
\end{align}
By \eqref{eq-bounded-of-surface-integral-e-v-over-rho-2---},
\begin{equation}
\label{eq-general-of-bounded-of-surface-integral-e-v-over-rho-2-}
\left|\int_{|\sigma|=1}e^{v(\rho\sigma)}\,d\sigma\right|
\leq\frac{C}{\rho^2},\qquad \forall\rho>0
\end{equation}
for some constant $C>0$. Hence by \eqref{ratios-difference-eqn} and 
\eqref{eq-general-of-bounded-of-surface-integral-e-v-over-rho-2-},
\begin{equation}\label{eq-estimate-for-I-1}
|I_1|\le C\int_{2|x|}^{\infty}\frac{|x|}{\rho^{n-1}}\cdot
\frac{1}{\rho^2}\cdot\rho^{n-1}\,d\rho=C'<\infty\quad\forall x\in\R^n.
\end{equation}
for some constant $C'>0$. By 
\eqref{eq-general-of-bounded-of-surface-integral-e-v-over-rho-2-},
\begin{equation}\label{eq-aligned-estimate-for-I-2}
\begin{aligned}
I_3&=\frac{1}{n(n-2)\omega_n}\int_{0<|y|\leq 1}\frac{e^{v(y)}}{|y|^{n-2}}\,dy
+\frac{1}{n(n-2)\omega_n}\int_{1\leq |y|\leq 2|x|}\frac{e^{v(y)}}{|y|^{n-2}}\,dy\\
&\le C\int_{0}^{1}\frac{1}{\rho^{n-2}}\cdot\rho^{n-1}\,d\rho
+C\int_1^{2|x|}\frac{1}{\rho^{n-2}}\cdot\frac{1}{\rho^2}\cdot\rho^{n-1}\,d\rho\\
&=C(1+\log(2|x|))\\
&\le C'\log|x|\qquad\qquad\forall |x|\ge 2. 
\end{aligned}
\end{equation}
On the other hand by \eqref{e^v-bd}, 
\begin{equation}\label{I2-upper-bd}
|I_2|\le\int_{D_1(x)}\frac{A_1}{|x-y|^{n-2}|y|^2}\,dy
+\int_{D_2(x)}\frac{A_1}{|x-y|^{n-2}|y|^2}\,dy=I_{2,1}+I_{2,2}
\end{equation}
where
\begin{equation}\label{D1-D2-defn}
\begin{cases}
D_1(x)=\{y\in\R^n:|y|\leq 2|x|\quad \mbox{and}\quad |x-y|\leq \frac{|x|}{2}\}\\
D_2(x)=\{y\in\R^n:|y|\leq 2|x|\quad \mbox{and}\quad |x-y|\geq \frac{|x|}{2}\}.
\end{cases}
\end{equation}
Since
\begin{equation*}
|x-y|\leq \frac{|x|}{2} \quad \Rightarrow \quad |y|\geq \frac{|x|}{2},
\end{equation*}
we have
\begin{equation}\label{eq-aligned-estimate-for-I-2-2-1}
I_{2,1}\leq \frac{C}{|x|^2}\int_{|x-y|\leq \frac{|x|}{2}}\frac{1}{|x-y|^{n-2}}\,dy
=C'<\infty\quad\forall x\ne 0
\end{equation}
for some constant $C'>0$. Finally on $D_2(x)$,
\begin{equation}\label{eq-aligned-estimate-for-I-2-2-2}
I_{2,2}\leq \frac{C}{|x|^{n-2}}\int_{|y|\leq 2|x|}\frac{1}{|y|^{2}}\,dy=C''
<\infty\quad\forall x\ne 0
\end{equation}
for some constant $C''>0$. By \eqref{w1-split-eqn}, 
\eqref{eq-estimate-for-I-1}, \eqref{eq-aligned-estimate-for-I-2}, 
\eqref{I2-upper-bd}, \eqref{eq-aligned-estimate-for-I-2-2-1} 
and \eqref{eq-aligned-estimate-for-I-2-2-2}, we get 
\eqref{eq-bounded-of-w-1-by-log-function} and lemma follows.
\end{proof}

\begin{lemma}\label{lem-v-+-w-1-+-w-2-=constant}
The following holds.
\begin{equation*}
v(x)+(\alpha-n\beta)w_1(x)+\beta w_2(x)=v(0)\quad\forall\R^n.
\end{equation*}
\end{lemma}
\begin{proof}
Let $q=v+(\alpha-n\beta)w_1+\beta w_2$. Then by 
\eqref{eq-cases-aligned-laplace-of-w1-and-w2} and 
Lemma \ref{lem-convergence-of-w2-w-r-t-R-to-infty}, 
\begin{equation*}
\La q=0\quad\mbox{ in }\R^n\quad\mbox{ in the distribution sense.} 
\end{equation*}
Let $x_0,x_1\in\R^n$. By \eqref{e^v-bd} there exists a constant $R_0>0$ 
such that
\begin{equation}\label{v-negative-eqn}
v(x)<0 \qquad \forall |x|\ge R_0.
\end{equation}
By the mean value theorem for harmonic functions,
\begin{equation}\label{eq-express-of-difference-after-using-MVT-of-harmonic}
q(x_0)-q(x_1)=\frac{1}{|B_R|}\left(\int_{B_{R}(x_0)}q\,dx
-\int_{B_{R}(x_1)}q\,dx\right) \qquad \forall R>0.
\end{equation}
Let $a=|x_0-x_1|$ and $R>R_0+2+2a+|x_0|+|x_1|$. Since $B_{R_0}\subset 
B_{R-2a}(x_1)\subset B_{R-a}(x_0)\subset B_{R}(x_1)$, 
\begin{equation}\label{sets-inclusion-relation}
\R^n\setminus B_{R-a}(x_0)\subset\R^n\setminus B_{R-2a}(x_1)\subset
\R^n\setminus B_{R_0}.
\end{equation}
Then by \eqref{v-negative-eqn}, \eqref{sets-inclusion-relation} and
Lemma \ref{lem-limit-of-v-as-rho-to-infty},
\begin{equation}
\label{eq-aligned-estimate-of-difference-of-v-between-two-points}
\begin{aligned}
&\frac{1}{|B_R|}\left(\int_{B_{R}(x_0)}v\,dx-\int_{B_{R}(x_1)}v\,dx\right)\\
&\qquad =\frac{1}{|B_R|}\left(\int_{B_{R}(x_0)\setminus B_{R-a}(x_0)}v\,dx
-\int_{B_{R}(x_1)\setminus B_{R-a}(x_0)}v\,dx\right)\\
&\qquad \le-\frac{1}{|B_R|}\int_{B_{R}(x_1)\setminus B_{R-a}(x_0)}v\,dx\\
&\qquad \le-\frac{1}{|B_R|}\int_{B_{R}(x_1)\setminus B_{R-2a}(x_1)}v\,dx\\
&\qquad \leq-\frac{1}{|B_R|}\left(\inf_{R-2a\leq\rho\leq R}
\left[\frac{1}{\log\rho}\int_{|\sigma|=1}v(x_1+\rho\sigma)
\,d\sigma\right]\right)\int_{R-2a}^{R}\rho^{n-1}\log\rho\,d\rho\\
&\qquad \leq -\frac{2aR^{n-1}\log R}{|B_{R}|}\left(\inf_{R-2a\leq\rho\leq R}
\left[\frac{1}{\log\rho}\int_{|\sigma|=1}v(x_1+\rho\sigma)\,d\sigma\right]
\right)\\
&\qquad \to 0 \qquad \qquad \mbox{ as }R\to\infty.
\end{aligned}
\end{equation} 
On the other hand, since $\left(B_{R}(x_0)\bs B_{R}(x_1)\right)\cup
\left(B_{R}(x_1)\bs B_{R}(x_0)\right)\subset\left(B_{R+a}(x_0)\bs B_{R-a}(x_0)
\right)$,
\begin{equation}\label{eq-aligned-seperate-of-int-of-w-1-w-2-into-J-1-J-2}
\begin{aligned}
&\frac{1}{|B_R|}\left(\int_{B_R(x_0)}\left((\alpha-n\beta)w_1+\beta w_2\right)
\,dx-\int_{B_R(x_1)}\left((\alpha-n\beta)w_1+\beta w_2\right)\,dx\right)\\
&\qquad \qquad \leq \frac{|\alpha-n\beta|}{|B_R|}
\int_{B_{R+a}(x_0)\bs B_{R-a}(x_0)}|w_1|\,dx+\frac{|\beta|}{|B_R|}
\int_{B_{R+a}(x_0)\bs B_{R-a}(x_0)}|w_2|\,dx\\
&\qquad \qquad :=|\alpha-n\beta|J_1+|\beta|J_2.
\end{aligned}
\end{equation}
By Lemma \ref{lem-estimate-for-w-1-by-I-1-and-I-2},
\begin{equation}\label{eq-estimate-of-J-1-final-daufhdsihe}
J_1\le\frac{C_1}{\left|B_R\right|}\int_{R-a\leq|x-x_0|\leq R+a}\log |x|\,dx
\leq \frac{2C_1a\log(|x_0|+R+a)|\partial B_{R+a}|}{|B_R|}\to 0\qquad 
\mbox{as $R\to\infty$} 
\end{equation}
for some constant $C_1>0$. We next observe that by 
\eqref{eq-aligend-w-2-after-integration-by-part},
\begin{align}\label{J2-split-eqn}
J_2\le&\frac{1}{n\omega_n|B_R|}\int_{B_{R+a}(x_0)\bs B_{R-a}(x_0)}
\left|\int_{|y|>2|x|}\left(\frac{(x-y)\cdot y}{|x-y|^n}+\frac{1}{|y|^{n-2}}
\right)e^{v(y)}\,dy\right|\,dx\notag\\
&\quad +\frac{1}{n\omega_n|B_R|}\int_{B_{R+a}(x_0)\bs B_{R-a}(x_0)}\left[
\int_{|y|\le 2|x|}\frac{|y|}{|x-y|^{n-1}}e^{v(y)}\,dy
+\int_{|y|\le 2|x|}\frac{e^{v(y)}}{|y|^{n-2}}\,dy\right]\,dx\notag\\
:=&J_{2,1}+J_{2,2}.
\end{align}
By \eqref{eq-estimate-of-F-by-using-Taylor-expansion-F1} and
\eqref{eq-general-of-bounded-of-surface-integral-e-v-over-rho-2-},
\begin{align}
\label{eq-estimate-of-J-2-1-by-absolutevalue-of-ball-R-a}
J_{2,1}\leq&\frac{C}{|B_R|}\int_{B_{R+a}(x_0)\bs B_{R-a}(x_0)}|x|\left(
\int_{2|x|}^{\infty}\frac{1}{\rho^{n-1}}\cdot\frac{1}{\rho^2}
\cdot\rho^{n-1}\,d\rho\right)dx\notag\\
=&\frac{C}{2|B_R|}\left|B_{R+a}(x_0)\bs B_{R-a}(x_0)\right|\notag\\
\leq&C'\frac{(R+a)^{n-1}}{R^n}\to 0 \qquad \mbox{as $R\to\infty$}.
\end{align}
Let
$$
\4{I}_1=\int_{|y|\leq 2|x|}\frac{|y|}{|x-y|^{n-1}}e^{v(y)}\,dy\qquad\mbox{ and }
\qquad \4{I}_2=\int_{|y|\leq 2|x|}\frac{e^{v(y)}}{|y|^{n-2}}\,dy.
$$
Then by the proof of Lemma \ref{lem-estimate-for-w-1-by-I-1-and-I-2},
\begin{equation}\label{eq-aligned-estimate-of-the-first-term-of-J-2-1-1-}
\4{I}_2\le C\log|x|\qquad\forall |x|\ge 2.
\end{equation}
On the other hand by \eqref{e^v-bd},
\begin{align}\label{eq-esstimate-of-quantity-G-34}
\4{I}_1\le&\int_{|y|\leq 2|x|}\frac{A_1}{|x-y|^{n-1}|y|}\,dy
=\int_{D_1(x)}\frac{A_1}{|x-y|^{n-1}|y|}\,dy
+\int_{D_2(x)}\frac{A_1}{|x-y|^{n-1}|y|}\,dy\notag\\
=&\4{I}_{1,1}+\4{I}_{1,2}
\end{align}
where $D_1(x)$, $D_2(x)$, are as given by \eqref{D1-D2-defn}.
Since $|y|\ge |x|/2$ for all $y\in D_1(x)$, 
\begin{equation}\label{eq-aligned-estimate-for-J-2-2-1}
\4{I}_{1,1}\leq \frac{C}{|x|}\int_{|x-y|\leq \frac{|x|}{2}}\frac{1}{|x-y|^{n-1}}\,dy
=C'<\infty
\end{equation}
and
\begin{equation}\label{eq-aligned-estimate-for-J-2-2-2}
\4{I}_{1,2}\leq\frac{2^{n-1}A_1}{|x|^{n-1}}\int_{|y|\leq 2|x|}\frac{1}{|y|}\,dy=C''
<\infty
\end{equation}
for some constants $C'>0$, $C''>0$. By 
\eqref{eq-aligned-estimate-of-the-first-term-of-J-2-1-1-}, 
\eqref{eq-esstimate-of-quantity-G-34},
\eqref{eq-aligned-estimate-for-J-2-2-1} and  
\eqref{eq-aligned-estimate-for-J-2-2-2},
\begin{equation}\label{eq-estimate-of-J-2-2-by-absolutevalue-of-ball-R-a}
\begin{aligned}
J_{2,2}\le&\frac{C}{|B_R|}\int_{B_{R+a}(x_0)\setminus B_{R-a}(x_0)}(1+\log |x|)\,dx\\
\le&C(1+\log(|x_0|+R+a))\frac{\left|B_{R+a}(x_0)
\bs B_{R-a}(x_0)\right|}{|B_R|}\\
\le&C'\frac{(R+a)^{n-1}\log (a+|x_0|+R)}{R^n}\to 0 \qquad 
\mbox{as $R\to \infty$}.
\end{aligned}
\end{equation}
Letting $R\to 0$ in 
\eqref{eq-express-of-difference-after-using-MVT-of-harmonic}, by 
\eqref{eq-aligned-estimate-of-difference-of-v-between-two-points}, 
\eqref{eq-aligned-seperate-of-int-of-w-1-w-2-into-J-1-J-2}, 
\eqref{eq-estimate-of-J-1-final-daufhdsihe},
\eqref{J2-split-eqn},
\eqref{eq-estimate-of-J-2-1-by-absolutevalue-of-ball-R-a} and 
\eqref{eq-estimate-of-J-2-2-by-absolutevalue-of-ball-R-a},
we get
\begin{equation*}
q(x_0)-q(x_1)\le 0 \qquad \forall x_0,\,x_1\in\R^n.
\end{equation*}
Since $x_0$, $x_1$ are arbitrary, by interchanging the roles of $x_0$ 
and $x_1$ in the above argument we get
\begin{equation*}
q(x_1)-q(x_0)\le 0 \qquad \forall x_0,\,x_1\in\R^n.
\end{equation*}
Hence
\begin{equation*}
q(x_0)-q(x_1)=0\qquad \forall x_0,\,x_1\in\R^n.
\end{equation*}
Thus $q$ is a constant. Hence $q(x)=q(0)=v(0)$ for any $x\in\R^n$ 
and the lemma follows.
\end{proof}

\begin{lemma}\label{lem-estimate-of-w-1-with-log-function-1}
\begin{equation*}
\frac{w_1(x)}{\log|x|}\to \frac{A_0}{n\omega_n} \qquad \mbox{as $|x|\to\infty$}.
\end{equation*}
\end{lemma}
\begin{proof}
Let $I_1$, $I_2$, and $I_3$ be the same as the proof of 
Lemma \ref{lem-estimate-for-w-1-by-I-1-and-I-2}. Then $I_1$ satisfies
\eqref{eq-estimate-for-I-1}. By the proof of 
Lemma \ref{lem-estimate-for-w-1-by-I-1-and-I-2} there exists a constant 
$C_1>0$ such that
\begin{equation}\label{I2-bd-eqn}
|I_2|\le C_1\quad\forall x\in\R^n.
\end{equation}
By Lemma \ref{lem-limit-of-v-as-rho-to-infty} there exists a constant 
$R_0>1$ such that \eqref{eq-bounded-of-surface-integral-e-v-over-rho-2---}
holds with $\3=1/3$. Now
\begin{equation}\label{I3-defn}
\begin{aligned}
I_3&=\frac{1}{n(n-2)\omega_n}\int_{|y|\leq |x|}\frac{e^{v(y)}}{|y|^{n-2}}
\,dy
+\frac{1}{n(n-2)\omega_n}\int_{|x|\leq |y|\leq 2|x|}\frac{e^{v(y)}}{|y|^{n-2}}
\,dy\\
&=I_{3,1}+I_{3,2}.
\end{aligned}
\end{equation}
Since by \eqref{eq-bounded-of-surface-integral-e-v-over-rho-2---} for any 
$|x|>R_0$,
\begin{equation}\label{I31-to-infty-eqn}
I_{3,1}\ge\frac{1}{n(n-2)\omega_n}\int_{R_0}^{|x|}\rho
\int_{|\sigma|=1}e^{v(\rho\sigma)}\,d\sigma\,d\rho
\ge C\int_{R_0}^{|x|}\frac{d\rho}{\rho}=C\log\left(\frac{|x|}{R_0}\right)
\to\infty\quad\mbox{as }|x|\to\infty, 
\end{equation}
by 
\eqref{e^v-integral-mean} and 
l'Hospital rule,
\begin{equation}
\label{eq-aligned-estimate-for-w-2-over-log-x-with-splitting-into-four-2-09}
\lim_{|x|\to\infty}\frac{I_{3,1}}{\log |x|}=\frac{\lim_{|x|\to\infty}|x|^2
\int_{|\sigma|=1}e^{v(|x|\sigma)}\,d\sigma}{n(n-2)\omega_n}=\frac{A_0}{n\omega_n}.
\end{equation}
By \eqref{eq-bounded-of-surface-integral-e-v-over-rho-2---},
\begin{equation}
\label{I-32-bd}
I_{3,2}=\frac{1}{n(n-2)\omega_n}
\int_{|x|}^{2|x|}\rho\left(\int_{|\sigma|=1}e^{v(\rho\sigma)}\,d\sigma\right)
\,d\rho\leq C_2\int_{|x|}^{2|x|}\frac{d\rho}{\rho}=C_2\log 2<\infty
\quad \forall |x|>R_0,
\end{equation}
for some constant $C_2>0$. By \eqref{I3-defn}, 
\eqref{eq-aligned-estimate-for-w-2-over-log-x-with-splitting-into-four-2-09}
and \eqref{I-32-bd}, 
\begin{equation}\label{I3-limit}
\lim_{|x|\to\infty}\frac{I_3}{\log |x|}=\frac{A_0}{n\omega_n}.
\end{equation}
Hence by \eqref{eq-estimate-for-I-1}, \eqref{I2-bd-eqn} and \eqref{I3-limit},
we get
\begin{equation*}
\lim_{|x|\to\infty}\frac{w_1(x)}{\log |x|}=\frac{A_0}{n\omega_n}
\end{equation*}
and lemma the follows. 
\end{proof}

\begin{lemma}\label{lem-estimate-of-w-2-with-log-function-1}
\begin{equation}\label{eq-result-inequality-of-lemma-in-section-4-with-w-2}
\frac{w_2(x)}{\log|x|}\to \frac{(n-2) A_0}{n\omega_n} \qquad 
\mbox{as $|x|\to\infty$}.
\end{equation}
\end{lemma}
\begin{proof}
By \eqref{eq-aligend-w-2-after-integration-by-part},
\begin{equation}\label{w2-3parts}
\begin{aligned}
w_2(x)&=\frac{1}{n\omega_n}\int_{|y|\leq 2|x|}\frac{(x-y)\cdot y}{|x-y|^n}
\cdot e^{v(y)}\,dy
+\frac{1}{n\omega_n}\int_{|y|>2|x|}\left(
\frac{(x-y)\cdot y}{|x-y|^n}+\frac{1}{|y|^{n-2}}\right)\,e^{v(y)}\,dy\\
&\qquad +\frac{1}{n\omega_n}\int_{|y|\leq 2|x|}\frac{e^{v(y)}}{|y|^{n-2}}
\,dy\\
&=I_1+I_2+(n-2)I_3.
\end{aligned}
\end{equation}
Since
$$
|I_1|\le \frac{1}{n\omega_n}\int_{|y|\leq 2|x|}\frac{|y|}{|x-y|^{n-1}}e^{v(y)}\,dy,
$$
by the proof of 
Lemma \ref{lem-v-+-w-1-+-w-2-=constant} there exists a constant $C>0$ such that
\begin{equation}\label{I1-uniform-bd}
|I_1|\le C\quad\forall x\in\R^n.
\end{equation}
By \eqref{eq-estimate-of-F-by-using-Taylor-expansion-F1} and
\eqref{eq-general-of-bounded-of-surface-integral-e-v-over-rho-2-},
\begin{equation}
\label{eq-aligned-estimate-for-w-2-over-log-x-with-splitting-into-four-1}
|I_2|\le C|x|\int_{|y|\geq 2|x|}\frac{e^{v(y)}}{|y|^{n-1}}\,dy
\le C'|x|\int_{2|x|}^{\infty}\frac{d\rho}{\rho^2}=C''<\infty
\end{equation}
for some constants $C'>0$, $C''>0$. By the proof of Lemma \ref{lem-estimate-of-w-1-with-log-function-1}, 
$I_3$ satisfies \eqref{I3-limit}.
By \eqref{I3-limit}, \eqref{w2-3parts}, \eqref{I1-uniform-bd} and
\eqref{eq-aligned-estimate-for-w-2-over-log-x-with-splitting-into-four-1}, 
we get 
\eqref{eq-result-inequality-of-lemma-in-section-4-with-w-2} and the lemma 
follows.
\end{proof}

\noindent We are now ready for the proof of Theorem \ref{v-limit-thm}.

\begin{proof}[\textbf{Proof of Theorem \ref{v-limit-thm}}]
By Lemma \ref{lem-v-+-w-1-+-w-2-=constant}, 
Lemma \ref{lem-estimate-of-w-1-with-log-function-1} and 
Lemma \ref{lem-estimate-of-w-2-with-log-function-1}, 
\begin{equation*}
\lim_{|x|\to\infty}\frac{v(x)}{\log|x|} =-L
\end{equation*}
where
\begin{equation*}
L=\frac{(\alpha-2\beta)A_0}{n\omega_n}.
\end{equation*}
Suppose that $L>2$. Then there exist constants $a\in (2,n)$ and $R_1>1$
such that
\begin{equation*}
e^{v(x)} \leq |x|^{-a}\qquad \forall |x|\geq R_1.
\end{equation*} 
Hence
\begin{align*}
&0<\frac{1}{\log R}\int_1^{R}\frac{1}{\rho^{n-1}}\left(\int_{B_{\rho}}e^{v}\,dy
\right)\,d\rho \leq \frac{C}{\log R}, \qquad \forall R>R_1\\
&\qquad \Rightarrow \lim_{R\to\infty}\frac{1}{\log R}\int_1^{R}\frac{1}{\rho^{n-1}}
\left(\int_{B_{\rho}}e^{v}\,dy\right)\,d\rho=0\quad\mbox{ as }R\to\infty 
\end{align*}
which contradicts \eqref{e^v-integral-cond}. Hence $L\le 2$.
Suppose that $L<2$. Then there exist constants $b\in (0,2)$ and $R_2>1$ such 
that
\begin{equation}\label{e^v-lower-bd}
e^{v(x)} \geq |x|^{-b}\qquad \forall |x|\geq R_2.
\end{equation} 
By \eqref{e^v-lower-bd} and a direct computation,
\begin{align*}
&\frac{1}{\log R}\int_1^{R}\frac{1}{\rho^{n-1}}\left(\int_{B_{\rho}}e^{v}\,dy
\right)\,d\rho \geq \frac{C_1R^{2-b}}{\log R}-\frac{C_2}{\log R}\qquad 
\forall R>R_2\\
\Rightarrow\quad&\lim_{R\to\infty}\frac{1}{\log R}\int_1^{R}\frac{1}{\rho^{n-1}}
\left(\int_{B_{\rho}}e^{v}\,dy\right)\,d\rho=\infty\quad\mbox{ as }R\to\infty
\end{align*} 
which again contradicts \eqref{e^v-integral-cond}. Hence $L=2$ and 
the theorem follows.
\end{proof}

\begin{cor}\label{A0-value-cor}
If $v$ is a solution of \eqref{v-eqn} which satisfies 
\eqref{e^v-integral-cond} and 
\eqref{e^v-bd}, then 
\begin{equation*}
\alpha>2\beta \qquad \mbox{and} \qquad A_0=\frac{2n\omega_n}{\alpha-2\beta}.
\end{equation*}
\end{cor}

\section{Radial symmetry of the solution}
\setcounter{equation}{0}
\setcounter{thm}{0}

In this section we will prove that under some condition on the solution
$v$ of \eqref{v-eqn}, $v$ is radially symmetric about the origin.
We first start with a proposition.

\begin{prop}\label{rv'-infty-prop}
Let $n\ge 3$, $\alpha>2\beta$, $a_0\in\R$, and let $v$ be the unique 
radially symmetric solution of \eqref{v-eqn} with $v(0)=a_0$ 
constructed in \cite{Hs2}. Then \eqref{e^v-limit} holds and
\begin{equation*}
\lim_{r\to\infty}rv'(r)=-2.
\end{equation*}
\end{prop}
\begin{proof}
We first observe that \eqref{e^v-limit} is proved in \cite{Hs2}. 
We next observe that by putting $x_0=0$ in \eqref{v-integral-eqn} 
and simplifying,
\begin{equation*}
rv'(r)=\frac{n\beta-\alpha}{r^{n-2}}\int_0^r\rho^{n-1}e^{v(\rho)}\,d\rho 
-\beta r^2e^{v(r)}.
\end{equation*}
Hence by \eqref{e^v-limit},
\begin{align*}
\lim_{r\to\infty}rv'(r)=&\lim_{r\to\infty}\frac{n\beta-\alpha}{r^{n-2}}
\int_0^r\rho^{n-1}e^{v(\rho)}\,d\rho -\beta\lim_{r\to\infty}r^2e^{v(r)}\\
=&(n\beta-\alpha)\lim_{r\to\infty}\frac{r^{n-1}e^{v(r)}}{(n-2)r^{n-3}}
-\beta\lim_{r\to\infty}r^2e^{v(r)}\\
=&\left[\frac{n\beta-\alpha}{n-2}-\beta\right]\lim_{r\to\infty}r^2e^{v(r)}\\
=&-2.
\end{align*}
\end{proof}

\begin{prop}\label{cor-behaviour-of-x-cdot-nabla-v-at-infty-d}
Let $n\ge 3$, $\alpha>2\beta$, and let $v$ be a solution of \eqref{v-eqn} 
which satisfies \eqref{e^v-integral-cond}, \eqref{e^v-bd} and 
\eqref{grad-v-bd-assumption}. Then
\begin{equation*}
x\cdot\nabla v\to -2 \qquad \mbox{uniformly as $|x|\to\infty$}. 
\end{equation*}
\end{prop}
\begin{proof}
Let
\begin{equation*}
V(x)=x\cdot\nabla v(x).
\end{equation*}
By \eqref{v-eqn} and a direct computation $V$ satisfies
\begin{equation}\label{eq-aligned-La-V-=-F}
\begin{aligned}
\La V=F(x)\qquad\mbox{ in }\R^n
\end{aligned}
\end{equation}
where
\begin{equation}\label{F-defn}
F(x)\equiv -\left(2\alpha+(\alpha+2\beta) V+\beta V^2\right)e^{v}
-\beta(x\cdot\nabla V)e^{v}.
\end{equation} 
Let $0<\mu<1$ and $x\in\R^n$ be such that $|x|>1$. For any $y\in 
B_{\frac{|x|}{2}}(x)$, let 
$$
d_y=\mbox{dist}\left(y,\partial B_{\frac{|x|}{2}}(x)\right)
\qquad\mbox{ and }\qquad R=\frac{1}{3}d_y.
$$ 
Then by Green's representation formula, 
\begin{equation}\label{Green-representation}
V(z)=H(z)+N(z) \qquad \forall z\in B_{2\mu R}\left(y\right)
\end{equation}
for some harmonic function $H$ in $B_{2\mu R}\left(y\right)$ 
where 
\begin{equation}\label{N(x)-formula}
N(z)=\frac{1}{n(2-n)\omega_n}\int_{B_{2\mu R}(y)}|z-w|^{2-n}F(w)\,dw
\end{equation}
is the Newtonian potential of $F$ in $B_{2\mu R}(y)$. Since 
$d_z\ge d_y/3=R$ for any $z\in B_{2\mu R}(y)$, by 
\eqref{Green-representation}, \eqref{N(x)-formula}, and Theorem 2.1
and Lemma 4.1 of \cite{GT},
\begin{align*}
d_y|\nabla N(y)|=&\frac{d_y}{n(n-2)\omega_n}\left|
\int_{B_{2\mu R}(y)}\nabla_y(|y-w|^{2-n})F(w)\,dw\right|\\
\le&C_1\mu d_y^2\sup_{B_{2\mu R}(y)}\left|F\right|\\
\le&C_1\mu\|d_z^2F(z)\|_{L^{\infty}(B_{|x|/2}(x))}
\end{align*}
and
\begin{align*}
d_y|\nabla H(y)|\le&\frac{C_2d_y}{\mu R}\sup_{B_{2\mu R}(y)}|H|\le\frac{3C_2}{\mu}
\left(\sup_{B_{2\mu R}(y)}|V|+\mu^2R^2\sup_{B_{2\mu R}(y)}|F|\right)\\
\le&\frac{3C_2}{\mu}\left(\|V\|_{L^{\infty}(B_{|x|/2}(x))}
+\mu^2\left\|d_z^2F(z)\right\|_{L^{\infty}(B_{|x|/2}(x))}\right)
\end{align*}
for some constants $C_1>0$, $C_2>0$. Hence
\begin{equation}\label{d_y-nabla-V-upper-bd1}
\|d_y\nabla V(y)\|_{L^{\infty}(B_{|x|/2}(x))}\le\frac{C_3}{\mu}
\left(\|V\|_{L^{\infty}(B_{|x|/2}(x))}
+\mu^2\|d_z^2F(z)\|_{L^{\infty}(B_{|x|/2}(x))}\right)
\end{equation}
for some constant $C_3>0$. Since $d_z\le |x|/2\le |z|$ for any 
$z\in B_{|x|/2}(x)$, 
\begin{equation}\label{d_z^2-bd}
d_z^2e^{v(z)}\le A_1\quad\forall z\in B_{|x|/2}(x).
\end{equation}
Then by \eqref{e^v-bd}, \eqref{F-defn} and \eqref{d_z^2-bd}, 
\begin{equation}\label{d_zF-upper-bd}
|d_z^2F(z)|\le C_4(1+\|V\|_{L^{\infty}(B_{|x|/2}(x))}+\|V\|_{L^{\infty}(B_{|x|/2}(x))}^2
+\|d_w\nabla V(w)\|_{L^{\infty}(B_{|x|/2}(x))})
\end{equation} 
for some constant $C_4>0$ and any $z\in B_{|x|/2}(x)$. Hence by 
\eqref{d_y-nabla-V-upper-bd1} and \eqref{d_zF-upper-bd},
\begin{align}\label{d_y-nabla-V-upper-bd2}
&\|d_y\nabla V(y)\|_{L^{\infty}(B_{|x|/2}(x))}\notag\\
&\quad \le \frac{C_5}{\mu}(1+\|V\|_{L^{\infty}(B_{|x|/2}(x))}+\|V\|^2_{L^{\infty}(B_{|x|/2}(x))})
+\mu^2\|d_y\nabla V(y)\|_{L^{\infty}(B_{|x|/2}(x))}
\end{align}
for some constant $C_5>0$. We now choose $\mu=1/(2C_5)$. Then by 
\eqref{d_y-nabla-V-upper-bd2},
\begin{equation*}
\frac{|x|}{4}\|\nabla V\|_{L^{\infty}(B_{|x|/4}(x))}
\le\|d_y\nabla V(y)\|_{L^{\infty}(B_{|x|/2}(x))}
\le 2C_5\left(1+\|V\|_{L^{\infty}(\R^n)}+\|V\|^2_{L^{\infty}(\R^n)}\right).
\end{equation*}
Hence
\begin{equation}\label{eq-for-uniform-convergence-of-x-cdot-nabla-v-3}
\|\nabla V\|_{L^{\infty}(B_{\frac{|x|}{4}}\left(x\right))}
\leq \frac{C_6}{|x|}\left(1+\|V\|_{L^{\infty}(\R^n)}+\|V\|^2_{L^{\infty}(\R^n)}\right)
\end{equation}
where $C_6=8C_5$. Let $r>1$. Taking supremum over all $|x|\ge r$ in 
\eqref{eq-for-uniform-convergence-of-x-cdot-nabla-v-3}, we get
\begin{equation}\label{eq-for-uniform-convergence-of-x-cdot-nabla-v-394857869}
\|\nabla V\|_{L^{\infty}(\R^n\bs B_{r})}\leq \frac{C_6}{r}
\left(1+\|V\|_{L^{\infty}(\R^n)}+\|V\|^2_{L^{\infty}(\R^n)}\right).
\end{equation}
We next choose $\3_0>0$ such that for any $\sigma$, $\sigma'\in S^{n-1}$, 
the line segment $l_{\sigma,\sigma'}$ 
joining $\sigma$ and $\sigma'$ is outside $B_{\frac{1}{2}}$, i.e., 
$|\xi|\geq \frac{1}{2}$ for all $\xi\in l_{\sigma,\sigma'}$. Then by 
\eqref{eq-for-uniform-convergence-of-x-cdot-nabla-v-394857869}, 
$\forall x=r\sigma$, $x'=r\sigma'$, $|\sigma-\sigma'|<\epsilon_0$, 
\begin{align*}
|x\cdot\nabla v(x)-x'\cdot\nabla v(x')|=&|V(x)-V(x')|\\
=&\left|\int_0^1\frac{\1}{\1 t}V(tx+(1-t)x')\,dt\right|\\
\le&|x-x'|\cdot\sup_{0\le t\le 1}|\nabla V(tx+(1-t)x')|\\
\le&|\sigma-\sigma'|\cdot\sup_{0\le t\le 1}\frac{C}{t\sigma+(1-t)\sigma'}\\
\le&C|\sigma-\sigma'|
\end{align*}
for some constant $C>0$. Hence the family $\{rv_r(r\sigma)\}_{r>1}$ is 
equi-H\"older continuous on $S^{n-1}$. Let $\{r_i\}_{i=1}^{\infty}$ be a sequence
of positive numbers such that $r_i\to\infty$ as $i\to\infty$. Then 
$\{r_{i}\}_{i=1}^{\infty}$ has a subsequence which we may assume 
without loss of generality to be the sequence itself that converges 
uniformly on $S^{n-1}$ as $i\to\infty$. Then by the l'Hospital rule and
Theorem \ref{v-limit-thm},
\begin{equation*}
\lim_{i\to\infty}r_iv_{r_i}(r_i\sigma)
=\lim_{i\to\infty}\frac{v(r_i\sigma)}{\log r_i}=-2 \qquad 
\mbox{ uniformly on $S^{n-1}$ as $i\to\infty$}.
\end{equation*}
Since the sequnece $\{r_i\}_{i=1}^{\infty}$ is arbitrary, $rv_r(r\sigma)\to-2$ 
uniformly on $S^{n-1}$ as $r\to\infty$ and the proposition follows.
\end{proof}

\noindent We will now prove Theorem \ref{radial-symmetric-thm}.

\begin{proof}[\textbf{Proof of Theorem \ref{radial-symmetric-thm}}]
Since any rotation in $\R^n$ can be decomposed into a finite number of 
rotations in 2-dimensional planes, it suffices to prove that 
\begin{equation*}
\Phi_{12}\equiv 0 \qquad \mbox{in $\R^n$}.
\end{equation*}
By direct computation $\Phi_{12}$ satisfies
\begin{equation*}
\La q+\alpha e^{v}q+\beta\left(x\cdot\nabla v\right)e^{v}q
+\beta\left(x\cdot\nabla q\right)e^{v}=0.
\end{equation*}
Let $w(x)=\Phi_{12}(x)/g(x)$ where $g(x)=|x|^{2-n}$. Then $w$ satisfies
\begin{equation*}\label{eq-satisfied-by-overline-v-sub-theta}
\La w+\left(\frac{2\nabla g}{g}+\beta e^{v}x\right)\cdot\nabla w
+\left(\alpha+\beta(x\cdot\nabla v)
+\frac{\beta(x\cdot\nabla g)}{g}\right)e^{v}w=0\quad\mbox{ in }\R^n.
\end{equation*}
If $\beta\le 0$, then $\alpha<n\beta\le 2\beta$. Then by 
Corollary \ref{non-existence-cor} \eqref{v-eqn} has no solution and
contradiction arises. Hence $\beta>0$. We now choose $\3>0$ such 
that $\alpha-n\beta+\epsilon\beta<0$. By 
Proposition \ref{cor-behaviour-of-x-cdot-nabla-v-at-infty-d} there 
exists a constant $R_0>0$ such that  
\begin{equation*}
x\cdot\nabla v<-2+\3\qquad \forall |x|\geq R_0.
\end{equation*}
Then
\begin{equation*}
\begin{aligned}
\alpha+\beta(x\cdot\nabla v)+\frac{\beta(x\cdot\nabla g)}{g}
\le\alpha-n\beta+\3\beta<0 \qquad\forall |x|\geq R_0
\end{aligned}
\end{equation*}
Suppose $v$ is radially symmetric in $B_{R_0}$. Then $\Phi_{12}\equiv 0$ in 
$\overline{B_{R_0}}$. Hence
\begin{equation*}
w\equiv 0 \qquad \mbox{in $\overline{B_{R_0}}$}.
\end{equation*}
We next observe that by \eqref{radial-symm-assum-infty},
\begin{equation*}
|w|\leq |x|^{n-2}|\Phi_{12}(x)|\to 0\qquad \mbox{as $|x|\to \infty$}.
\end{equation*}
Then by the maximum principle for $w$ in $R^n\bs\overline{B_{R_0}}$, 
\begin{equation*}
w\equiv 0\qquad \mbox{in $\R^n\bs \overline{B_{R_0}}$}.
\end{equation*}
Hence
\begin{equation*}
\Phi_{12}\equiv 0 \qquad \mbox{in $\R^n$}.
\end{equation*}
Thus $v$ is radially symmetric in $\R^n$.
\end{proof}

\noindent {\bf Acknowledgement} Sunghoon Kim supported by Priority Research Centers Program through the National Research Foundation of Korea(NRF) funded by the Ministry of Education, Science and Technology(2012047640).

\end{document}